\documentclass[10pt]{amsart}
\usepackage{a4}
\usepackage{amssymb}
\usepackage{array}
\usepackage{booktabs}
\usepackage{multirow}
\usepackage{hhline}
\usepackage{graphicx}

\newtheorem{thm}{Theorem}

\newtheorem{cor}[thm]{Corollary}
\newtheorem{lem}[thm]{Lemma}
\newtheorem{prop}[thm]{Proposition}
\newtheorem{rem}[thm]{Remark}

\numberwithin{thm}{section}
\numberwithin{equation}{section}
\newcommand{\Real}{\mathbb R}

\newcommand{\norm}[1]{\left\Vert#1\right\Vert}
\newcommand{\abs}[1]{\left\vert#1\right\vert}

\newcommand{\eps}{\varepsilon}

\newcommand{\la}{\langle}
\newcommand{\ra}{\rangle}

\newcommand{\A}{\mathcal{A}}

\newcommand{\E}{\mathbb{E}}

\newcommand{\F}{\mathcal{F}}

\newcommand{\Hi}{\mathcal{H}}
\newcommand{\I}{\mathcal{I}}

\newcommand{\M}{\mathcal{M}}

\newcommand{\N}{\mathcal{N}}

\newcommand{\Om}{\Omega}

\begin{document}
\title{Distance $k$-graphs of hypercube and $q$-Hermite polynomials}
\author{Hun Hee Lee}
\author{Nobuaki Obata}

\address{Hun Hee Lee :
Department of Mathematics, Chungbuk National University,
410 Sungbong-Ro, Heungduk-Gu, Cheongju 361-763, Korea}
\email{hhlee@chungbuk.ac.kr}

\address{Nobuaki Obata:
Graduate School of Information Sciences, Tohoku University,
Sendai, 980-8579, Japan}
\email{obata@math.is.tohoku.ac.jp}



\begin{abstract}
We will prove that some weighted graphs on the distance $k$-graph of hypercubes approximate the $q$-Hermite polynomial of a $q$-gaussian variable by providing an appropriate matrix model.
\end{abstract}

\maketitle

\section{Introduction}
During the last decade quantum probabilistic approach to the study of asymptotic spectral analysis of graphs 
has been developed considerably, 
and many important probability distributions are obtained as scaled limits of spectral distributions
of growing graphs, see \cite{HO07} and references cited therein.
For example, the spectral distribution of the Hamming graph $H(d,n)$ converges to the standard normal distribution 
or the Poisson distribution by taking a suitable scaling limit as $d,n\to \infty$, which was
first proved by means of quantum probabilistic method in \cite{HOT}.

In this paper we extend quantum probabilistic method to weighted graphs
and derive the distribution of the $q$-gaussian variable $G_q$ $(-1\le q\le 1)$ by Bo\.zejko and Speicher \cite{BS91}. 
In fact, we will show that
the probability distribution of $H^q_k(G_q)$,
where $H^q_k$ is the $k$-th $q$-Hermite polynomial,
is derived from a sequence of weighted graphs on the distance $k$-graph of $n$-cube.
It is noteworthy that $q$-gaussian variables for $q\neq 0,1$ have not been so far observed 
in line with the asymptotic spectral analysis \cite{HO07},
where consideration has been mostly 
restricted to simple graphs, i.e.,  undirected graphs with no loops and no weights (multi-edges).
While, our result in a particular case of  $q=1$ 
is proved in \cite{Obata12} in an alternative manner.

This paper is organized as follows. In section \ref{sec-Preliminaries} we will recall some basic notations and tools we need for the proof of the result,
which includes the baby Fock model for $q$-gaussians. In section \ref{sec-hypercube} we will explain the connection between the distance $k$-graph of hypercubes and the baby Fock model, and finally we will prove the main result in section \ref{sec-main}.
For basic concepts of quantum (non-commutative) probability, see e.g., \cite{HO07, NS06}.

\section{Preliminaries}\label{sec-Preliminaries}

\subsection{Non-commutative probability and convergence of distributions}

Let $(\M_n, \varphi_n)_{n \ge 1}$ and $(\M, \varphi)$ be non-commutative (tracial) $W^*$-probability spaces. For $m\ge 1$ we consider self-adjoint elements
$X^n_1, \cdots, X^n_m \in \M_n$ and $X_1, \cdots, X_m\in \M$.
We say that the sequence of $n$-tuples $(X^n_1, \cdots, X^n_m)$ 
{\it converges to $(X_1, \cdots, X_m)$ in distribution} as $n\to \infty$ if
	$$ \varphi_n(P(X^n_1 ,\cdots, X^n_m)) \to \varphi(P(X_1 ,\cdots, X_m))$$
for any polynomial $P$ of $m$ non-commuting variables. We will use the notation
	$$(X^n_1, \cdots, X^n_m) \stackrel{\text{\rm dist}}{\longrightarrow} (X_1, \cdots, X_m).$$

\begin{rem}{\rm
In classical probability theory, we say that a sequence of random variables $X_n$ converges to $X$ in distribution if the probability distribution of $X_n$ converges weakly to that of $X$. Therefore the above notion of convergence in distribution for non-commuting random variables is not compatible to the one in classical probability theory. We say often that a sequence of random variables $X_n$ converges to $X$ in moments  if the $m$-th moment of $X_n$ converges to that of $X$ for all $m$. It is known that the convergence in moments implies the convergence in distribution if the distribution of $X$ allows a positive answer to the determinate moment problem. For more information along this line, see e.g., \cite[Chapter 4]{Chung}.
}
\end{rem}

We present a lemma, which is probably a folklore in non-commutative probability. We include the proof for the convenience of the readers. Recall that 
the non-commutative $L^p$-space $L^p(\M, \varphi)$ is the completion of $(\M, \norm{\cdot}_p)$, where the $p$-norm on $\M$, $\norm{\cdot}_p$ is defined by
	$$\norm{x}_p = \left(\varphi(\abs{x}^p)\right)^{\frac{1}{p}},\;\; x\in \M.$$ 
	
\begin{lem}\label{lem-conv-dist}
Let $X^n_1, \cdots, X^n_m \in \M_n$ and $X_1, \cdots, X_m\in \M$ be self-adjoint elements
and assume that 
\[
(X^n_1, \cdots, X^n_m) \stackrel{\text{\rm dist}}{\longrightarrow} (X_1, \cdots, X_m).
\]
Let $Y^n_1, \cdots, Y^n_m \in \M_n$ be self-adjoint elements satisfying
	\begin{enumerate}
		\item
		$\{\norm{X^n_i}_p, \norm{Y^n_i}_p: n\ge 1, 1\le i \le m\}$ are bounded for all $p\ge 1$ with bounds depending only on $p$.
		\item
		$\norm{X^n_i - Y^n_i}_p \to 0$ as $n\to \infty$ for all $p\ge 1$.
	\end{enumerate}
Then, we also have
	$$(Y^n_1, \cdots, Y^n_m) \stackrel{\text{\rm dist}}{\longrightarrow} (X_1, \cdots, X_m).$$
\end{lem}
\begin{proof}
It is enough to prove the case of monomials. Let us fix a monomial $P$. Our strategy is the following. We begin with $\lim_{n\to \infty}\varphi_n(P(X^n_1 ,\cdots, X^n_m))$ and exchange each $X^n_1$ with $Y^n_1$ one by one. Here we use the H\"older inequality and the conditions (1) and (2). Then we replace $X^n_2$ by $Y^n_2$ similarly, and we repeat the same procedure until we replace all $X^n_i$'s by $Y^n_i$'s.

Let us do the case of $\varphi((X^n_1)^2(X^n_2)^3)$ as an example.
Then, by the H\"older inequality we have
	$$\abs{\varphi((X^n_1 - Y^n_1)X^n_1(X^n_2)^3)} \le \norm{X^n_1 - Y^n_1}_5\norm{X^n_1}_5\norm{X^n_2}^3_5,$$
so that the conditions (1) and (2) allow us to conclude that
	$$\lim_{n\to \infty}\varphi((X^n_1)^2(X^n_2)^3) = \lim_{n\to \infty}\varphi(Y^n_1X^n_1(X^n_2)^3).$$
Note that the number 5 in $\norm{\cdot}_5$ is the degree of the monomial.

\end{proof}

\subsection{The $q$-gaussians}

In this subsection we recall the $q$-gaussian variables. Let $H_\Real$ be a real Hilbert space and $\Hi = H_\Real + i H_\Real$ be its complexification.
We consider the operator of symmetrization $P_n$ on $\Hi^{\otimes n}$ defined by
	$$P_0 \Om = \Om,\;\; P_n (f_1\otimes \cdots \otimes f_n) = \sum_{\pi \in S_n}q^{i(\pi)}f_{\pi(1)}\otimes \cdots \otimes f_{\pi(n)},$$
where $S_n$ denotes the symmetric group of permutations of $n$ elements and
    	$$i(\pi) = \# \{(i,j) | 1\leq i,j \leq n, \pi(i) > \pi(j)\}$$
is the number of inversions of $\pi \in S_n$.

Now we define the {\it $q$-inner product} $\left\langle \cdot, \cdot \right\rangle_q$ on $\Hi^{\otimes n}$ by
    $$\left\langle \xi, \eta \right\rangle_q =\left\langle
    \xi, P_n \eta \right\rangle \;\, \text{for}\;\, \xi, \eta\in
    \Hi^{\otimes n},$$
where $\left\langle \cdot, \cdot \right\rangle$ is the inner product in $\Hi^{\otimes n}$. When $-1<q<1$ $P_n$'s are strictly positive (\cite{BS91}), so that $\left\langle \cdot, \cdot \right\rangle_q$ is actually an inner product. We denote by $\Hi^{\otimes_q n}$ the resulting Hilbert space. Then one can associate the $q$-Fock space $\F_q(\Hi)$.
    $$\F_q(\Hi) = \mathbb{C}\Om \oplus \bigoplus_{n\geq 1}\Hi^{\otimes_q n},$$ where $\Om$ is a unit vector called vacuum. 
When $q=0$ we recover the so-called full Fock space over $\Hi$. 
In the extreme cases $q = \pm 1$, we have
\[
\F_1(\Hi)= \mathbb{C}\Om \oplus \bigoplus_{n\geq 1}\Hi^{\otimes_s n},
\qquad
\F_{-1}(\Hi) = \mathbb{C}\Om \oplus \bigoplus_{n\geq 1}\Hi^{\otimes_a n},
\]
which are referred to as the Bosonic and Fermionic Fock spaces, respectively.
Here $\otimes_s$ and $\otimes_a$ refer to symmetric and anti-symmetric tensor product of Hilbert spaces, respectively.

For $h\in H_\Real$, we can define a generalized $q$-semi-circular random variable on $\F_q(\Hi)$ ($-1<q<1$) by
	$$s(h)=\ell_q(h)+\ell^*_q(h),$$
where $\ell_q(h)$ is the left creation operator by $h\in \Hi$ and $\ell^*_q(h)$ is the adjoint of $\ell_q(h)$. Recall that $\ell_q(h)$ is given by
	$$\ell_q(h) \Om = h,\;\; \ell_q(h) f_1\otimes \cdots \otimes f_n = h\otimes f_1\otimes \cdots \otimes f_n.$$
Note that in the extreme cases $q=\pm 1$ the creation operators have 
slightly different forms due to the symmetrization procedure. For example, we have
	$$\ell_1(h) \Om = h,\;\; \ell_1(h) f^{\otimes n} = \frac{1}{n+1}\sum^n_{k=0} f^{\otimes n-k} \otimes h \otimes f^{\otimes k}.$$
Then, we get a $W^*$-probability space $(\Gamma_q, \varphi_q)$, where $\Gamma_q$ is the von Neumann algebra generated by $\{s(h): h\in H_\Real\}$ in $B(\F_q(\Hi))$ and $\varphi_q(\cdot) = \la \, \cdot\, \Om, \Om\ra$ is the vacuum state.

\subsection{Baby Fock space, central limit procedure and hypercontractivity}\label{BabyFock}

In this section we collect background materials focusing on baby Fock space and its central limit procedure due to Biane \cite{Bi97}.

Let $I = \{1, 2, \cdots, n\}$ be a fixed index set and $\eps : I\times I\rightarrow \{\pm 1\}$ be a ``choice of sign" function satisfying
	$$\eps(i,j) = \eps(j,i),\; \eps(i,i)=-1,\; \forall i,j\in I.$$
Now we consider the unital algebra $\A(I,\eps)$ with generators $(x_i)_{i\in I}$ satisfying
	$$x_ix_j - \eps(i,j)x_jx_i = 2\delta_{i,j},\;\; i,j\in I.$$
In particular, we have $x^2_i = 1$, $i\in I$, where $1$ refers to the unit of the algebra. 
The involution on $\A(I,\eps)$ is defined by $x^*_i=x_i$.
We will use the following notations for the elements in $\A(I,\eps)$.
	$$x_{\emptyset} := 1 \;\;\text{and}\;\; x_A := x_{i_1}\cdots x_{i_k},\;
	A = \{i_1< \cdots < i_k\} \subseteq I.$$
Then, $\{x_A : A\subseteq I\}$ is a basis for $A(\I,\eps)$. Let $\phi^\eps : \A(I,\eps) \rightarrow \mathbb{C}$ be the tracial state given by
\[
\phi^\eps(x_A) = \delta_{A,\emptyset},
\]
which gives rise to a natural inner product on $\A(I,\eps)$ as follows:
\[
\la x, y \ra := \phi^\eps(y^*x),\;\, x,y\in \A(I,\eps).
\]
Let
	$$H = L^2(\A(I,\eps),\phi^\eps)$$
be the corresponding $L^2$-space, then clearly $\{x_A : A\subseteq I\}$ is an orthonormal basis for $H$.

Now we consider left creations $\beta^*_i$ and left annihilations $\beta_i$ in $B(H)$, $i\in I$ in this context.
	$$\beta^*_i(x_A) = 
	\left\{\begin{array}{ll}x_ix_A & \text{if}\; i\notin A\\0 & \text{if}\; i\in A \end{array}\right.,
	\;\;\beta_i(x_A) = 
	\left\{\begin{array}{ll}x_ix_A & \text{if}\; i\in A\\0 & \text{if}\; i\notin A \end{array}\right.,\;\;
	i\in I,\; A\subseteq I.$$
The baby gaussians are defined as follows.
	$$\gamma_i := \beta^*_i + \beta_i,\;\; 1\le i\le n.$$
It is straightforward to check that $\gamma_i$ is the same as left multiplication operator with respect to $x_i$, so that we often identify $\gamma_i$ and $x_i$. Let $\Gamma_n$ be the von Neumann algebra generated by $\{\gamma_i\}^n_{i=1}$ in $B(H)$. It is also known that $1$ is a cyclic and separating vector for $\Gamma_n$, and the above is the faithful GNS representation of $(\Gamma_n, \tau_n^\eps)$, where $\tau_n^\eps$ is the vacuum state on $\Gamma_n$ given by
	$$\tau_n^\eps(\cdot) = \la\, \cdot\, 1, 1\ra.$$
Note that $\tau_n^\eps$ is tracial on $\Gamma_n$.

Now we replace $\{\eps(i, j) : 1\le i < j \le n\}$ with a family of i.i.d. random variables with
	$$P(\eps(i,j)= -1) = \frac{1-q}{2}, \;\; P(\eps(i,j)= 1) = \frac{1+q}{2}.$$
We set
	$$s_n = \frac{1}{\sqrt{n}}\sum^n_{i=1}\gamma_i.$$
Then, the Speicher's central limit procedure (\cite{Sp92} or \cite{Bi97}) tells us the following.
\begin{prop}\label{prop-q-CLT}
We have
\[
s_n \stackrel{\text{\rm dist}}{\longrightarrow} G_q
\]
for almost every $\eps$.
In other words, we have for any polynomial $Q$
\[
\lim_{n\rightarrow \infty}\tau^\eps_n(Q(s_n)) = \tau_q(Q(G_q))
\]
for almost every $\eps$.
\end{prop}

We close this section with the last ingredient for the proof, namely hypercontractivity of baby Ornstein-Uhlenbeck semigroup. Recall that the number operator $N_\eps$ on $H$ is given by $N_\eps = \sum_{i\in I} \beta^*_i\beta_i$. Then, for any $A \subseteq I$ we have
$N_\eps x_A = \abs{A}x_A$. Since $1 \in H$ is separating and cyclic for 
$\Gamma_n$ we define the $\eps$-Ornstein-Uhlenbeck semigroup $P^\eps_t : \Gamma_n \rightarrow \Gamma_n$ by
	\begin{equation}\label{eps-OU-semigp}
	P^\eps_t(X)1 = e^{-tN_\eps}(X1),\;\; X\in \Gamma_n.
	\end{equation}

\begin{thm}\label{thm-hyper}
Let $1< p < r <\infty$ we have
			$$\norm{P^\eps_t}_{L^p \rightarrow L^r} \le 1\;\; \text{if and only if}\;\; e^{-2t} \le \frac{p-1}{r-1}.$$
\end{thm}

The above result is due to Biane \cite[Theorem 5]{Bi97},
which leads us to a Khinchine type inequality as follows.

\begin{cor}\label{cor-Khinchine}
	Let $\displaystyle X = \sum^n_{i_1, \cdots, i_k = 1} \alpha_{i_1\cdots i_k} \gamma_{i_1}\cdots\gamma_{i_k}\in \Gamma_n$, then we have
		$$\norm{X}_2 \le \norm{X}_p \le (p-1)^{\frac{k}{2}}\norm{X}_2.$$
	\end{cor}
\begin{proof}
Note that $P^\eps_t(X) = e^{-kt}X$. Then it is a direct application of Theorem \ref{thm-hyper}.
\end{proof}

\section{The distance $k$-graph of hypercubes and Baby Fock model}\label{sec-hypercube}

We begin this section by recalling some graph theoretic notions.
A {\it weighted graph} consists of the set of vertices $V$, the of edges $E \subset V\times V$ and a collections of weights $A = (a_{xy})_{(x,y)\in E}$, $a_{xy} \in \Real$. We say that $(V, E, A)$ is a weighted graph on the unweighted graph $(V, E)$ (i.e. every weight is identically 1). The matrix $A$ is called the {\it adjacency matrix} of the weighted graph. When the graph is finite (i.e. finite number of vertices) $A$ is a non-commutative random variable in a $W^*$-probability space $(M_{\abs{V}}, \varphi_e)$, where $M_{\abs{V}}$ is the matrix algebra of the size $\abs{V}$ and $\varphi_e(B) =\la B \delta_e, \delta_e \ra = B_{ee}$ for a fixed point $e \in V$ and $B\in M_{\abs{V}}$.

{\it The $n$-cube} is the unweighted graph $K^n_2 = K_2\times \cdots \times K_2$, the $n$-fold direct product of the complete graph $K_2$ with two vertices. Then there is a bijection between the vertices of $K^n_2$ and the set of pre-described orthonormal basis of $H$ as follows.
	$$(r_1,\cdots, r_n) \mapsto x^{r_1}_1\cdots x^{r_n}_n,$$
where $r_i \in \{0,1\}$ and $x^0_i = 1$. Note that $x^{r_1}_1\cdots x^{r_n}_n = x_A$ with $A=\{1\le i\le n : r_i = 1\}$.
 
In $n$-cube there is an edge between two vertices $(r_1,\cdots, r_n)$ and $(s_1,\cdots, s_n)$ if and only if there is only one index $i$ such that $r_i\ne s_i$. {\it The distance $k$-graph of $K^n_2$} ($1\le k\le n$) is the graph with the same set of vertices but with the set of edges described as follows.
	$$\{\left((r_1,\cdots,r_n), (s_1,\cdots, s_n)\right):
	\text{there are exactly $k$ indices $1\le i\le n$ with $r_i \ne s_i$}\}.$$

Now we focus on the the {\it Bosonic} case for the moment, which implies that $q=1$ and so that the choice of sign function $\eps$ is identically 1. Then, it is straightforward to see that the operator
	$$\sum^n_{i=1}\gamma_i$$
represents the adjacency matrix of $K^n_2$ since $\gamma_i$ is nothing but multiplying $x_i$. It is again straightforward to see that the operator
\[
\frac{1}{k!}\sum_{\substack{i_1,\cdots,i_k\\ \ne}}\gamma_{i_1}\cdots\gamma_{i_k} 
= \sum_{ i_1<i_2<\cdots<i_k }\gamma_{i_1}\gamma_{i_2}\cdots \gamma_{i_k}
\]
represents the adjacency matrix of the distance $k$-graph of $K^n_2$. 
Here the summation in the left-hand side is taken over distinct
$i_1,\cdots,i_k$ taken from $\{1,\cdots, n\}$.

Now let us turn our attention to the case of general $\eps$. This time $\gamma_i$ still acts on $x^{r_1}_1\cdots x^{r_n}_n$ by multiplying $x_i$ on the left, but we need to take the commutation relations into account. Thus, the operator
	$$\sum^n_{i=1}\gamma_i$$
corresponds to the adjacency matrix of the {\it weighted graph on $K^n_2$} described as follows: We put the weight
	$$\eps({\bf r}, i) = \eps(i,1)^{r_1}\cdots \eps(i,i-1)^{r_{i-1}}$$
on the edge between ${\bf r} = (r_1,\cdots, r_n)$ and ${\bf s} = (s_1,\cdots, s_n)$ with $s_i\ne r_i$, but $s_j = r_j$, $\forall j\ne i$. Similarly, the operator
	$$\frac{1}{k!}\sum_{\substack{i_1,\cdots,i_k\\ \ne}}\gamma_{i_1}\cdots\gamma_{i_k}$$
represents the adjacency matrix of a {\it weighted graph on the distance $k$-graph of $K^n_2$}. The rule for assigning weights on each edge is the following. If there is an edge between ${\bf r} = (r_1,\cdots, r_n)$ and ${\bf s} = (s_1,\cdots, s_n)$, then we have $k$ distinct indices $\{j_1<\cdots <j_k\}$ with $s_{j_l}\ne r_{j_l}$, $1\le l \le k$ but $s_j = r_j$, $\forall j\notin \{j_1, \cdots, j_k\}$. For a permutation $\sigma \in S_k$ and ${\bf j} = (j_1,\cdots, j_k)$ we define $\epsilon(\sigma, {\bf j})$ to be the number given by
	$$\gamma_{j_{\sigma(1)}}\cdots \gamma_{j_{\sigma(k)}} = \epsilon(\sigma, {\bf j})\gamma_{j_1}\cdots \gamma_{j_k}.$$
On the above edge we put the weight
	$$\frac{1}{k!} \sum_{\sigma\in S_k} \epsilon(\sigma, {\bf j})\cdot \prod^k_{l=1}\eps({\bf r}, j_l).$$

\section{Weighted graphs on Hypercubes and $q$-Hermite Polynomials}\label{sec-main}

In this section we focus on the analytic part of our results, namely the convergence analysis.

We recall the normalized $q$-Hermite polynomial $H^q_k$ given by the following recurrence relations.
	$$\begin{cases} H^q_0(x) = 1,\; H^q_1(x) = x\\
	xH^q_k(x) = H^q_{k+1}(x) + [k]_qH^q_{k-1}(x),\;\; k\ge 1,\end{cases}$$
where $[k]_q = \frac{1-q^k}{1-q}$.

We are interested in the following operator.
	$$X_{n,k} := \frac{1}{n^{\frac{k}{2}}}\sum_{\substack{i_1,\cdots,i_k\\ \ne}}\gamma_{i_1}\cdots\gamma_{i_k}.$$

We would like to find a recurrence relation regarding $X_{n,k}$ by multiplying $X_{n,1}$. Now we have
	\begin{align*}
	\lefteqn{\left( \sum_{\substack{i_1,\cdots,i_k\\ \ne}}\gamma_{i_1}\cdots\gamma_{i_k} \right) \left( \sum_i\gamma_i \right)}\\
	& = \sum_{\substack{i_1,\cdots,i_k, i\\ \ne}}\gamma_{i_1}\cdots\gamma_{i_k}\gamma_i
	+ \sum_{\substack{i=i_1,\cdots,i_k\\ \ne}}\gamma_{i_1}\cdots\gamma_{i_k}\gamma_i
	+ \cdots + \sum_{\substack{i_1,\cdots,i_k=i\\ \ne}}\gamma_{i_1}\cdots\gamma_{i_k}\gamma_i.
	\end{align*}
By relabeling we denote
$\displaystyle \sum_{\substack{i_1,\cdots,i_k, i\\ \ne}}\gamma_{i_1}\cdots\gamma_{i_k}\gamma_i = \sum_{\substack{i_1,\cdots,i_k, i_{k+1}\\ \ne}}\gamma_{i_1}\cdots\gamma_{i_{k+1}}$.
For the second term we have
	\begin{align*}
	\sum_{\substack{i=i_1,\cdots,i_k\\ \ne}}\gamma_{i_1}\cdots\gamma_{i_k}\gamma_i
	& = \sum_{\substack{i,i_2, \cdots,i_k\\ \ne}}\eps_{i,i_2}\cdots\eps_{i,i_k}\gamma_{i_2}\cdots\gamma_{i_k}\\
	& = \sum_{\substack{i,i_1, \cdots,i_{k-1}\\ \ne}}\eps_{i,i_1}\cdots\eps_{i,i_{k-1}}\gamma_{i_1}\cdots\gamma_{i_{k-1}}
	\end{align*}
by relabeling again. If we repeat the similar relabeling, then we get
	\begin{align}\label{eq-recur}
	\lefteqn{\left( \sum_{\substack{i_1,\cdots,i_k\\ \ne}}\gamma_{i_1}\cdots\gamma_{i_k} \right) \left( \sum_i\gamma_i \right)}\\
	& = \sum_{\substack{i_1,\cdots, i_{k+1}\\ \ne}}\gamma_{i_1}\cdots\gamma_{i_{k+1}}\nonumber\\
	& \;\;\;\; + \sum_{\substack{i_1,\cdots,i_{k-1}\\ \ne}}\sum_{i\notin (i_1, \cdots, i_{k-1})}(\eps_{i,i_1}\cdots\eps_{i,i_{k-1}} + \eps_{i,i_2}\cdots\eps_{i,i_{k-1}} + \cdots + 1)\gamma_{i_1}\cdots\gamma_{i_{k-1}}. \nonumber
	\end{align}

Now we define
	$$Y_{n,k} := \frac{1}{n^{\frac{k}{2}}}\sum_{\substack{i_1,\cdots,i_k\\ \ne}}[k+1]^{-1}_q Z(i_1,\cdots, i_k)\gamma_{i_1}\cdots\gamma_{i_k},$$
where
	$$Z(i_1,\cdots, i_k) := \frac{1}{n}\sum_{i\notin (i_1, \cdots, i_k)}(\eps_{i,i_1}\cdots\eps_{i,i_k} + \eps_{i,i_2}\cdots\eps_{i,i_k} + \cdots + 1).$$
Here comes the key recurrence relation obtained by dividing \eqref{eq-recur} with $n^{\frac{k+1}{2}}$.
	\begin{equation}\label{eq-recur2}
	X_{n,k+1} = X_{n,k}X_{n,1} - [k]_qY_{n,k-1},\;\; k\ge 1.
	\end{equation}
We will use Lemma \ref{lem-conv-dist} for the sequences $(X_{n,1}, \cdots, X_{n,k}, Y_{n,1}, \cdots, Y_{n,k-1})$ and  $(X_{n,1}, \cdots, X_{n,k}, X_{n,1}, \cdots, X_{n,k-1})$. In order to do so we need to check the conditions (1) and (2) in this case.

The condition (1) can be easily checked by Corollary \ref{cor-Khinchine}. Note that
	$$\{\gamma_{i_1}\cdots \gamma_{i_k} : \text{all distinct} \;\; i_1,\cdots,i_k \}$$
is an orthonormal family in $L^2(\Gamma_n)$ and the number of indices $(i_j)^k_{j=1}$ with all distinct entries is strictly smaller than $n^{\frac{k}{2}}$. 

The condition (2) is more involved. We first need to understand the limit behavior of $Z(i_1,\cdots,i_k)$ as $n\to \infty$.
Note that the sequence of random variables
	$$\{W_i = \eps_{i,i_1}\cdots\eps_{i,i_k} + \eps_{i,i_2}\cdots\eps_{i,i_k} + \cdots + 1:i\notin (i_1, \cdots, i_{k-1})\}$$
is an independent collection. Indeed for $i\ne i'$ both of them are not in $(i_1, \cdots, i_{k-1})$, we have $(i,i_j) \ne (i',i_{j'})$ and $(i,i_j) \ne (i_{j'}, i')$ for any $1\le j, j'\le k$. Moreover, each random variable has mean
	$$\E(W_i) = q^k + q^{k-1} + \cdots + 1  = [k+1]_q$$
and variance
	$$\sigma^2 = \E(W^2_i) - \E(W_i)^2 = k+1 + \sum^k_{j=1}2(k+1-j)q^k - [k+1]^2_q.$$
Thus, the classical central limit theorem tells us that
	$$\sqrt{n}([k+1]^{-1}_qZ(i_1,\cdots,i_k) - 1)\stackrel{\text{dist}}{\longrightarrow} \N(0,\sigma^2)$$
in our sense (i.e. convergence in moments, see \cite[section 30]{Billing95} or \cite{Bahr65}), so that their $L^p$-norms also converge. Thus, we have
	$$\norm{\sqrt{n}([k+1]^{-1}_qZ(i_1,\cdots,i_k) - 1)}_p \to \norm{\N(0,\sigma^2)}_p = \alpha_p,$$
and consequently
	\begin{equation}\label{eq-CLT-Lp-norm}
	\norm{[k+1]^{-1}_qZ(i_1,\cdots,i_k) - 1}_p \le \frac{\alpha_p+1}{\sqrt{n}}
	\end{equation}
for big enough $n$. Note that $\alpha_p$ is a constant depending only on $k$, $q$ and $p$, so that it is independent of $n$.

Now we get the estimate
	\begin{equation}\label{eq-Lp-norm-difference}
	\norm{X_{n,k} - Y_{n,k}}_p \le \frac{C_p}{\sqrt{n}}
	\end{equation}
for big enough $n$. Indeed, we have
	$$X_{n,k} - Y_{n,k} = \frac{1}{n^{\frac{k}{2}}}\sum_{\substack{i_1,\cdots,i_k\\ \ne}}(1-[k+1]^{-1}_qZ(i_1,\cdots,i_k))\gamma_{i_1}\cdots\gamma_{i_k}$$
and by Corollary \ref{cor-Khinchine}, Minkowski's inequality and \eqref{eq-CLT-Lp-norm} we have
	\begin{align*}
		\norm{X_{n,k} - Y_{n,k}}_p
		& = \left[\E\left(\norm{X^\eps_{n,k} - Y^\eps_{n,k}}_p\right)\right]^{\frac{1}{p}}\\
		& \le \frac{(p-1)^{\frac{k}{2}}}{n^{\frac{k}{2}}}\left[\E\left(\sum_{\substack{i_1,\cdots,i_k\\ \ne}}\abs{[k+1]^{-1}_qZ(i_1,\cdots,i_k) - 1}^2\right)^{\frac{p}{2}}\right]^{\frac{1}{p}}\\
		& \le \frac{(p-1)^{\frac{k}{2}}}{n^{\frac{k}{2}}}\left(\sum_{\substack{i_1,\cdots,i_k\\ \ne}}\norm{[k+1]^{-1}_qZ(i_1,\cdots,i_k) - 1}^2_p\right)^{\frac{1}{2}}\\
		& \le \frac{(p-1)^{\frac{k}{2}}(\alpha_p +1)}{\sqrt{n}},
	\end{align*}
where $X^\eps_{n,k}$ is the value of $X_{n,k}$ for a fixed choice of $\eps$. A standard application of Borel-Cantelli lemma with \eqref{eq-Lp-norm-difference} leads us to the following conclusion.
	$$\norm{X^\eps_{n,k} - Y^\eps_{n,k}}_p \to 0$$
for almost every $\eps$. This observation tells us, together with Proposition \ref{prop-q-CLT}, that we can choose a specific ``choice of sign" $\eps$ satsfying
	$$\begin{cases}X^\eps_{n,1} \stackrel{\text{dist}}{\longrightarrow} G_q\\
	\norm{X^\eps_{n,k} - Y^\eps_{n,k}}_p \to 0,\;\; \forall p\ge 1\end{cases}.$$
Note that we used the fact that the collection of all monomials in non-commuting variables is countable.
From now on we will fix this choice of $\eps$ and by abuse of notation we will denote $X^\eps_{n,k}$ and $Y^\eps_{n,k}$ simply by $X_{n,k}$ and $Y_{n,k}$, respectively. This explains how we get the condition (2).

Finally, we present the main convergence result.

	\begin{thm}\label{thm-main-conv}
	We have
		\begin{align*}
		\lefteqn{(X_{n,1}, \cdots, X_{n,k}, Y_{n,1}, \cdots, Y_{n,k-1})}\\
		& \stackrel{\text{dist}}{\longrightarrow} (H^q_1(G_q), \cdots, H^q_k(G_q), H^q_1(G_q), \cdots, H^q_{k-1}(G_q)).
		\end{align*}
	\end{thm}
\begin{proof}
We will use induction on $k$. When $k=1$ we get the result directly from Proposition \ref{prop-q-CLT}. When $k=2$ we note that
	$$X_{n,2} = \frac{1}{n}\sum_{i\ne j}\gamma_i\gamma_j = \left(\frac{1}{\sqrt{n}}\sum_i\gamma_i\right)^2 - I = X^2_{n,1} - I
	= H^q_2(X_{n,1}).$$
Thus, we clearly have
	$$(X_{n,1}, X_{n,2}, X_{n,1}) \stackrel{\text{dist}}{\longrightarrow} (H^q_1(G_q), H^q_2(G_q), H^q_1(G_q)).$$
If we compare two sequences $(X_{n,1}, X_{n,2}, X_{n,1})$ and $(X_{n,1}, X_{n,2}, Y_{n,1})$, then we get the result we wanted by Lemma \ref{lem-conv-dist} since we already checked the conditions (1) and (2) are satisfied. 

Now we suppose the the conclusion is true upto $k$ and let us check the case for $k+1$. Recall the recurrence relation
	$$X_{n,k+1} = X_{n,k}X_{n,1} - [k]_qY_{n,k-1},$$
so that we get
	\begin{align*}
	\lefteqn{(X_{n,1}, \cdots, X_{n,k}, X_{n,k+1}, Y_{n,1}, \cdots, Y_{n,k-1})}\\
	& \stackrel{\text{dist}}{\longrightarrow} (H^q_1(G_q), \cdots, H^q_k(G_q), H^q_{k+1}(G_q), H^q_1(G_q), \cdots, H^q_{k-1}(G_q)).
	\end{align*}
This can be trivially extended to the following convergence.
	\begin{align*}
	\lefteqn{(X_{n,1}, \cdots, X_{n,k}, X_{n,k+1}, Y_{n,1}, \cdots, Y_{n,k-1}, X_{n,k})}\\
	& \stackrel{\text{dist}}{\longrightarrow} (H^q_1(G_q), \cdots, H^q_k(G_q), H^q_{k+1}(G_q), H^q_1(G_q), \cdots, H^q_{k-1}(G_q), H^q_k(G_q)).
	\end{align*}
If we compare two sequences $(X_{n,1}, \cdots, X_{n,k}, X_{n,k+1}, Y_{n,1}, \cdots, Y_{n,k-1}, X_{n,k})$ and $(X_{n,1}, \cdots, X_{n,k}, X_{n,k+1}, Y_{n,1}, \cdots, Y_{n,k-1}, Y_{n,k})$, then we get the result we wanted by Lemma \ref{lem-conv-dist} since we already checked the conditions (1) and (2) are satisfied. This finishes the induction process.

\end{proof}

	\begin{cor}\label{cor-main}
	For any $k\ge 1$ we have
		$$X_{n,k} \stackrel{\text{dist}}{\longrightarrow}H^q_k(G_q).$$
	Thus, we can conclude that $H^q_k(G_q)$ can be approximated (in distribution) by a sequence of weighted graphs on the distance $k$-graph of $K^n_2$.
	\end{cor}

\begin{rem}{\rm
	\begin{enumerate}
		\item
		When $q=1$ we recover the result of N. Obata in \cite{Obata12}.
		\item
		The above result provides a matrix model for $H^q_k(G_q)$ which is homogeneous in degree.
	\end{enumerate}
}
\end{rem}

\bibliographystyle{amsplain}
\providecommand{\bysame}{\leavevmode\hbox
to3em{\hrulefill}\thinspace}

\end{document}